\documentclass[12pt]{amsart}

\usepackage{amsrefs}
\usepackage{amssymb}
\usepackage{esint}
\usepackage{mathtools}
\usepackage[english]{babel}
\usepackage{epsfig}
\usepackage{amsmath,amsfonts,amssymb}
\usepackage{mathtools}
\usepackage{enumitem}
\usepackage[all]{xy}
\usepackage{graphicx}
\usepackage{latexsym}
\usepackage{verbatim}
\usepackage{xcolor}
\usepackage[normalem]{ulem}
\usepackage{cancel}

\def\HYPER{\relax}

\ifdefined\HYPER
\usepackage[hidelinks, pdfstartview=FitB, pdfpagemode=UseNone]{hyperref}
\else
\renewcommand{\href}[2]{\relax}
\renewcommand{\url}[1]{#1}
\fi

\makeatletter
\@namedef{subjclassname@2020}{\textup{2020} Mathematics Subject Classification}
\makeatother

\usepackage[text={16.5cm, 22.5cm}, asymmetric]{geometry}
\numberwithin{equation}{section}

\def\v{\varphi}

\def\Re{\mathop{\mathrm{Re}}}
\def\Im{\mathop{\mathrm{Im}}}

\def\D{\mathbb D}

\def\R{\mathbb R}

\def\C{\mathbb C}

\def\Aut{{\sf Aut}}

\def\const{{\rm const}}

\DeclareMathOperator{\clD}{\mathrm{cl}_\UD}
\newcommand{\widesim}[1]{
  \mathrel{\lower.36ex\hbox{$\overset{#1}{\scalebox{1.25}[1]{$\sim$}}$}}
}
\newcommand{\hsegment}[1]{{[#1]_{h}}}

\newtheorem{theorem}{Theorem}[section]
\newtheorem{lemma}[theorem]{Lemma}
\newtheorem{proposition}[theorem]{Proposition}

\theoremstyle{definition}
\newtheorem{definition}[theorem]{Definition}
\newtheorem{example}[theorem]{Example}

\theoremstyle{remark}
\newtheorem{remark}[theorem]{Remark}
\newtheorem{problem}[theorem]{Problem}

\numberwithin{equation}{section}
\newtheorem*{remark*}{Remark}

\newcommand{\UD}{\mathbb{D}}
\newcommand{\UC}{\partial\UD}

\newcommand{\Real}{\mathbb{R}}
\newcommand{\Natural}{\mathbb{N}}
\newcommand{\di}{\mathrm{d}\hspace{.07em}}
\newcommand{\anglim}{\angle\lim}

\renewcommand{\emptyset}{\varnothing}
\renewcommand{\ge}{\geqslant}
\renewcommand{\le}{\leqslant}
\renewcommand{\geq}{\geqslant}
\renewcommand{\leq}{\leqslant}
\newcommand{\proofof}[1]{{\fontseries{bx}\fontshape{it}\selectfont Proof of #1}}

\newcommand{\StepPm}[1]{\medskip\noindent{\textsc{Proof of #1~}}}

\newenvironment{alphlist}{\begin{enumerate}[label={\bf (\alph*)}, ref={\rm (\alph*)}, left=0em]%
\everydisplay{\makeatletter\def\@eqnum{\normalfont(\theequation)}\makeatother}}{\end{enumerate}}
\newenvironment{romlist}{\begin{enumerate}[label={\bf (\hskip.07em\roman*\hskip.07em)}, ref={\rm (\hskip.05em\roman*\hskip.05em)}, left=-0.3em]%
\everydisplay{\makeatletter\def\@eqnum{\normalfont(\theequation)}\makeatother}}{\end{enumerate}}

\newcommand{\StepG}[1]{\medskip\noindent\textsc{#1.}}

\begin{document}
\title[Hyperbolic convexity of holomorphic level sets]{Hyperbolic convexity of holomorphic level sets}

\author[I. Efraimidis]{Iason Efraimidis$^{\,\dag}$}
\address{Iason Efraimidis: Departamento de Matem\'{a}ticas\\
	Universidad Aut\'{o}noma de Madrid\\
	28049 Madrid,
	Spain.}\email{iason.efraimidis@uam.es}

\thanks{$^\dag$The first author is partially supported by PID2019-106870GB-I00 from MICINN, Spain}

\author[P. Gumenyuk]{Pavel Gumenyuk}
\address{Pavel Gumenyuk: Dipartimento di Matematica\\ Politecnico di Milano\\ via E. Bonardi 9\\ Milan 20133, Italy.}
\email{pavel.gumenyuk@polimi.it}

\date\today

\dedicatory{Dedicated to the memory of  Prof. Christian Pommerenke}

\subjclass[2020]{Primary 30C80, 30F45, 52A55; secondary 30J99, 30H05, 51M10}

\begin{abstract}
We prove that the sublevel set $\big\{z\in\mathbb D\colon k_\UD\big(z,z_0\big)-k_\UD\big(f(z),w_0\big)<\mu\big\}$, ${\mu\in\mathbb R}$, is geodesically convex with respect to the Poincar\'e distance $k_\UD$ in the unit disk~$\mathbb D$ for every ${z_0,w_0\in\mathbb D}$ and  every holomorphic ${f:\mathbb D\to\mathbb D}$ if and only if ${\mu\le0}$. An analogous result is established also for the set $\{z\in\UD\colon 1-|f(z)|^2<\lambda(1-|z|^2)\}$, ${\lambda>0}$. This extends a result of Solynin~(2007) and solves a problem posed by Arango, Mej\'{\i}a and Pommerenke~(2019). We also propose several open questions aiming at possible extensions to more general settings.
\end{abstract}

\maketitle

\let\subs=\subset

\tableofcontents

\section{Introduction}
In 2007, Solynin~\cite{SolyninPAMS} proved a remarkable result, which in a slightly different form can be stated as follows.\\[1ex] \textit{Let $f:\UD\to\UD$ be a holomorphic self-map of the unit disk
  $\UD:={\{z\in\C\colon |z|<1\}}$
equipped with the Poincar\'e distance $k_\UD$, and let $z_0$ and $w_0$ be two points in~$\UD$. Then
$$
 \mathcal D(f)=\mathcal D(f;z_0,w_0):=\big\{z\in\UD\colon k_\UD(z,z_0)-k_\UD(f(z),w_0)<0\big\}
$$
is either empty or it is a hyperbolically convex domain, which means that for every ${z_1,z_2\in  \mathcal D(f)}$, the hyperbolic geodesic segment $\hsegment{z_1,z_2}$ joining these two points is entirely contained in~$\mathcal D(f)$.}
\medskip

Similarly to the classical Schwarz\,--\,Pick Lemma playing a fundamental role in Complex Analysis, both in one and higher dimensions, see  \textit{e.g.}~\cite{AbateNotes}, this result relates holomorphic mappings with the hyperbolic geometry in~$\UD$. To some extend, the interplay of hyperbolic convexity and holomorphicity has been previously exploited via J{\o}rgensen's Theorem~\cite{Joergensen}, see \textit{e.g.} \cite{Flinn, CarmPomm1997, ZipperAlg, BK2024}. It is also worth mentioning that hyperbolically convex functions, \textit{i.e.} conformal mappings onto hyperbolically convex domains, admit analytic characterization; their properties have been thoroughly studied, see \textit{e.g.}~\cite{MP2000, Texas, Rohde, MaKourou} and references therein. Recently, using Solynin's result stated above, it has been found in~\cite{Resolvents} that non-linear resolvents of infinitesimal generators are hyperbolically convex functions.

\smallskip

In this paper, we give a different proof and extend Solynin's result to two families of nested subsets~of~$\UD$:
$$
  \mathcal D_\mu(f)\subs\mathcal D_0(f)=\mathcal D(f)=\Omega_1(f)\subs\Omega_\lambda(f),\quad \mu<0,~\lambda>1.
$$

The sets $\mathcal D_\mu(f)$, ${\mu\in\Real}$, are defined in a very natural way, namely, as
$$
  \mathcal D_\mu(f):=\big\{z\in\UD\colon k_\UD(z,z_0)-k_\UD(f(z),w_0)<\mu\big\}.
$$
In Section~\ref{S_hyperbolic-distance-family}, we prove that for every $\mu<0$, this set is either empty or hyperbolically convex.

The definition of $\Omega_\lambda$ becomes more intuitive if we observe that replacing $f$ by the composition $T_2\circ f\circ T_1$, where $T_1$ and $T_2$ are suitable conformal automorphisms of $\UD$, we may suppose ${z_0=w_0=0}$ and accordingly rewrite ${k_\UD(f(z),w_0)>k_\UD(z,z_0)}$ in euclidian terms as ${|f(z)|>|z|}$, or equivalently as
$$
  \nu_f(z):=\frac{1-|f(z)|^2}{1-|z|^2}=\frac{\lambda_{\UD}(z)}{\lambda_{\UD}\big(f(z)\big)}<1,
$$
where $\lambda_\UD$ stands for the density of the Poincar\'e metric.

The result of Solynin stated above was conjectured (and proved for a special case) by Mej\'\i{a} and Pommerenke in their
study~\cite{MP2005} of the analytic fixed point function associated to~$f$,  which was shown to map~$\UD$ conformally onto $\mathcal D(f)=\Omega_1(f):={\{z\in\UD\colon \nu_f(z)<1\}}$; see also~\cite{Solynin2}.
Much later, Arango,  Mej\'\i{a}, and Pommerenke~\cite{AMP2019} considered the family
$$
 \Omega_\lambda(f):={\{z\in\UD\colon \nu_f(z)<\lambda\}},\quad {\lambda>0},
$$
and asked whether $\Omega_\lambda$ is necessarily hyperbolically convex for $\lambda>1$. In Section~\ref{S_lambda-family}, we give an affirmative answer to this question for all~${\lambda\ge 1}$. Surprisingly, this more general result is obtained by a simpler method as compared to the very elegant, but more involved ideas employed in~\cite{SolyninPAMS}.

It is also worth mentioning that $\mathcal D_\mu(f)$ fails to be hyperbolically convex if~${\mu>0}$ already for automorphisms (see Example~\ref{EX_mu>0}). Therefore, even though the interpretation of the level sets~$\Omega_\lambda(f)$ in terms of the hyperbolic geometry seems to be less direct than that of~$\mathcal D_\mu(f)$, these domains $\Omega_\lambda(f)$ appear to be a suitable way to extend  the family $\mathcal D_\mu(f)$ beyond the ``critical'' level set~$\mathcal D(f)$. Finally, note that the situation is symmetric: $\Omega_\lambda(f)$ are in general not hyperbolically convex for~${\lambda\in(0,1)}$, see Example~\ref{EX_lambda-family}.

The plan of the paper is as follows.
In the next section, we recall the basics of the hyperbolic geometry in~$\UD$ and prove some technical lemmas.
We prove that the domains $\Omega_\lambda(f)$, ${\lambda\ge1}$, and $\mathcal D_\mu(f)$, ${\mu<0}$, are hyperbolically convex in Sections~\ref{S_lambda-family} and~\ref{S_hyperbolic-distance-family}, respectively.
In the concluding Section~\ref{S_further}, we pose a few open questions indicating possible ways to develop similar results in various settings.

\section{Preliminaries and auxiliary results}
Recall that the hyperbolic (Poincar\'e) metric in~$\UD$ at a point ${z\in\UD}$ is given by $(u,v)\mapsto\lambda_\UD^2(z)\Re(u\overline v)$, ${u,v\in\C}$, where $\lambda_\UD(z):={2/(1-|z|^2)}$ is often referred to as the hyperbolic density in~$\UD$. This metric induces a  distance given by
$$
  k_\UD(z,w):=\log\frac{1+\rho_\UD(z,w)}{1-\rho_\UD(z,w)}, \quad\text{where~}~\rho_\UD(z,w):=\left|\frac{z-w}{1-\overline w z}\right|.
$$
The classical Schwarz\,--\,Pick Lemma, see  \textit{e.g.}~\cite[Theorem~3.2]{HBmetric1}, states that holomorphic self-maps of~$\UD$ do not increase the Poincar\'e distance~$k_\UD$, while its isometries are exactly the conformal automorphisms of~$\UD$, the set of which we denote by~$\Aut(\UD)$. In euclidian terms, the infinitesimal version of this fact can be expressed precisely as follows:
\begin{equation}\label{EQ_SchP_inf}
 |f'(z)|\le\frac{\lambda_\UD(z)}{\lambda_\UD(f(z))}=\frac{1-|f(z)|^2}{1-|z|^2}\,=:\,\nu_f(z)
\end{equation}
for all $z\in\UD$ and  every holomorphic function~${f:\UD\to\UD}$. The equality in~\eqref{EQ_SchP_inf} holds for some~${z\in\UD}$ if and only if~${f\in\Aut(\UD)}$, in which case equality holds for all~${z\in\UD}$. For further details, we refer interested readers to~\cite{HBmetric1,HBmetric2}.

\begin{remark}\label{RM_JWC}
As noticed in \cite[Section~2.2]{AMP2019}, the quantity $\nu_f(z)$ is also related to the boundary behavior of~$f$ and~$f'$. Namely, let ${\zeta\in\UC}$. Then, according to the classical Julia\,--\,Wolff\,--\,Carath\'eodory Theorem,
$$
\alpha_f(\zeta):=\liminf_{\UD\ni z\to\zeta}{\,\nu_f(z)}\,<\,+\infty
$$
if and only if the angular limits
$\anglim_{z\to\zeta}f(z){=:f(\zeta)}$, $\anglim_{z\to\zeta}f'(z)=:{f'(\zeta)}$
exist finitely and $f(\zeta)\in\UC$. Moreover, if these equivalent conditions hold, then ${\nu_f(z)\to\alpha_f(\zeta)}$ as ${z\to\zeta}$ non-tangentially, $\zeta\,\overline{f(\zeta)}\,f'(\zeta)\,=\,\alpha_f(\zeta)\,>\,0$, and
$$
\big|f(r\zeta)-f(\zeta)\big|<2\alpha_f(\zeta)\,(1-r)\quad\text{for all~$~r\in[0,1)$,}
$$
see \textit{e.g.} \cite[\S\S4.2--4.4]{ShapiroBook}, \cite[\S1.2]{Abate1} or \cite[Chapter~2]{Abate2}.
\end{remark}

Returning to the hyperbolic geometry of~$\UD$, we recall that geodesics in this geometry  are diameters of~$\UD$ and arcs of circles orthogonal to~$\UC$; as a consequence, for each pair of points ${z_1,z_2\in\UD}$ there is a unique geodesic segment $\hsegment{z_1,z_2}$ joining them. This gives a natural way to define an analogue of the classical euclidian notion of convexity in the context of the hyperbolic geometry:
\begin{definition}
A set $\Omega\subs\UD$ is said to be \textsl{hyperbolically convex}, or \textsl{h-convex} for short, if ${\hsegment{z_1,z_2}\subs\Omega}$ whenever ${z_1,z_2\in\Omega}$. If, in addition, every hyperbolic geodesic intersects $\partial\Omega$ at two points at most, then $\Omega$ is said to be \textsl{strictly h-convex}.
\end{definition}

We will make use of a simple observation connecting hyperbolic convexity to starlikeness. Throughout the paper, we will use the symbol ``$\subset$'' to denote the inclusion in the wide sense.
\begin{definition}
A set $\Omega\subs\C$ is said to be \textsl{starlike} if ${0\in\Omega}$ and for every straight line $L\subs\C$ passing through the origin, the intersection $L\cap\Omega$ is connected.
\end{definition}
\begin{lemma}\label{LM_starlike-h-convex}
A set $\Omega\subs\UD$ is h-convex if and only if for every $T\in\Aut(\UD)$ such that ${0\in T(\Omega)}$, the set $T(\Omega)$ is starlike.
\end{lemma}
\begin{proof}
The lemma follows easily from the observation that given two points ${z_1,z_2\in\Omega}$, ${z_1\neq z_2}$, the hyperbolic geodesic segment $\hsegment{z_1,z_2}$ is contained in~$\Omega$ if and only if the euclidian segment $[0,T(z_2)]$ is contained in~$T(\Omega)$ for some (and hence every) ${T\in\Aut(\UD)}$ with ${T(z_1)=0}$.
\end{proof}

We will also need the following technical lemma that borrows an idea from \cite[Sect.\,1.5]{Sugawa2005}.
\begin{lemma}\label{LM_starlike}
Let $D\subs\UD$ be a domain containing the origin. Suppose that $D$ is not starlike. Then there exists ${a_0\in\UD\cap\partial D}$ and ${\delta>0}$ such that the euclidian segment  ${\big[(1-\delta)a_0,(1+\delta)a_0\big]}$ is contained in the relative closure $\clD(D)$ of the domain~$D$ in~$\UD$.
\end{lemma}

\begin{proof}
Consider the radius function
$$
R(\omega) := \sup\{r>0 : [0,r\omega] \subs D \}, \qquad \omega\in\UC.
$$
It is not difficult to see that $R$ is lower semicontinuous. Indeed, if ${\omega_0\in\UC}$ and $0<r<R(\omega_0)$, then the straight line segment ${[0,r\omega_0]}$ is contained in~$D$. Since this segment is compact, it follows that
$$
 \{t\omega_0 e^{i\theta}\colon t\in[0,r],~\theta\in(-\epsilon,\epsilon)\big\}\,\subs\,D
$$
for some $\epsilon>0$. Therefore, ${R(\omega)>r}$ for all ${\omega\in\UC}$ sufficiently close to~$\omega_0$, as desired.

Moreover, if ${\omega_0\in\UC}$ and if every neighborhood of $R(\omega_0)\omega_0$ contains some points of ${\big\{r\omega_0:r\ge0\big\}\setminus\clD(D)}$, then the radius function is also upper semicontinuous at~$\omega_0$. To see this, we assume that $R(\omega_0)<1$ since it is otherwise evident. Let ${R(\omega_0)<r<1}$ and observe that, by assumption, there exists $\rho\in\big(R(\omega_0),r\big)$ for which $\rho\omega_0 \notin \clD(D)$. Since $\UD\backslash\clD(D)$ is an open set,  for all $\omega$ sufficiently close to~$\omega_0$ we have $\rho\omega \notin \clD(D)$ and hence ${R(\omega) < \rho < r}$, as desired.

Finally, noting that $R(\omega) \omega \in \partial D$ for every $\omega \in\UC$, the above argument shows that if the conclusion of the lemma fails, then the radius function is continuous on~$\UC$. Therefore, $\partial D$ is a starlike curve, in contradiction to the hypothesis.
\end{proof}

We complete this section with a lemma that allows us to reduce proving hyperbolic convexity to checking some local property at the boundary,  namely, the existence of supporting hyperbolic half-planes. Denote $$\UD(z_0,r):={\{z\in\C\colon |z-z_0|<r\}}.$$
\begin{lemma}\label{LM_main}
Let $u:\UD\to\R$ be a continuous function in~$\UD$ and let $\Omega$ be a connected component of~$\{z\in\UD\colon u(z)>0\}$. Suppose that for every point ${\zeta\in\UD\cap\partial\Omega}$, the function $u(z)$ is differentiable at~${z=\zeta}$ and satisfies the following two conditions:
\ifdefined\Pavel\begin{alphlist}\else\begin{enumerate}\fi
\item\label{IT_LM_main-grad-nonnullo}  $\nabla u(\zeta)\neq0$  and
\item\label{IT_LM_main-geodesic}  there exists $\varepsilon>0$ such that $\gamma_\zeta\,\cap\,\clD(\Omega)\,\cap\,\UD(\zeta,\varepsilon)=\{\zeta\}$, where $\gamma_\zeta$ stands for the hyperbolic geodesic passing through the point~$\zeta$ and orthogonal at that point to $\nabla u(\zeta)$.
\ifdefined\Pavel\end{alphlist}\else\end{enumerate}\fi
Then $\Omega$ is a strictly h-convex domain.
\end{lemma}
\begin{proof}
Clearly, $\Omega$ is a domain since it is a connected component of an open set. To show that~$\Omega$ is h-convex, suppose on the contrary that it is not. Then by Lemma~\ref{LM_starlike-h-convex}, there exists $T\in\Aut(\UD)$ such that the domain $D:={T\big(\Omega\big)}$ contains the origin but it is not starlike. By Lemma~\ref{LM_starlike}, it follows that $\partial D \cap \UD$ contains a point~$a_0$ such that
\begin{equation}\label{EQ_in-the-closure}
 \big[(1-\delta)a_0,(1+\delta)a_0\big]\,\subs\,\clD(D)
\end{equation}
for a sufficiently small $\delta>0$.

The preimage of $\big[(1-\delta)a_0,(1+\delta)a_0\big]$ with respect to $T$ is a segment of the hyperbolic geodesic~$\gamma$ passing through $z:={T^{-1}(0)\in\Omega}$ and $\zeta:={T^{-1}(a_0)\in\partial\Omega}$. According to condition~\ref{IT_LM_main-geodesic}, there exists a punctured neighbourhood~$U$ of the point~$\zeta$ such that the hyperbolic geodesic $\gamma_\zeta$  passing through~$\zeta$  orthogonally to $\nabla u(\zeta)$ does not have  common points with~$\clD(\Omega)\cap U$. Taking into account~\eqref{EQ_in-the-closure}, it follows that ${\gamma\neq\gamma_\zeta}$. Being two distinct hyperbolic geodesics, $\gamma$~and $\gamma_\zeta$ intersect at~$\zeta$ transversally. Taking into account that ${u(\zeta)=0}$ and ${\nabla u(\zeta)\neq0}$, the latter implies that in every neighbourhood of~$\zeta$ there are points of~$\gamma$ at which ${u<0}$. This, however, contradicts~\eqref{EQ_in-the-closure} because $u(w)\ge0$ for  every ${w\in \clD(\Omega)=T^{-1}(\clD(D)}$.

To complete the proof it remains to notice that if an h-convex domain is not strictly h-convex, then clearly its boundary contains a hyperbolic geodesic segment. However, for the domain~$\Omega$, this would contradict condition~\ref{IT_LM_main-geodesic}.
\end{proof}

\section{Level sets defined by the hyperbolic density}\label{S_lambda-family}
For ${\lambda>0}$ and a holomorphic self-map $f:\UD\to\UD$, we consider the lower level set
$$
\Omega_\lambda(f):=\Big\{z\in\UD\colon \frac{1-|f(z)|^2}{1-|z|^2}<\lambda\Big\}.
$$
We state the main result of this section as follows.
\begin{theorem}\label{TH_main1}
Let $\lambda\ge1$ and let $f:\UD\to\UD$ be holomorphic. In case ${\lambda=1}$, we also suppose that~${f(0)\neq0}$ and that ${f\not\in\Aut(\UD)}$.
Then the set $\Omega_\lambda(f)$ is a strictly hyperbolically convex domain containing the origin.
\end{theorem}
Our assumptions in Theorem~\ref{TH_main1} are sharp.
If ${f(0)=0}$ then, by the Schwarz Lemma, ${\Omega_1(f)=\emptyset}$. If ${f(0)\neq0}$ then $0\in\Omega_1(f)$ and, as Solynin~\cite{SolyninPAMS} proved, $\Omega_1(f)$ is strictly h-convex unless $f$ is an automorphism of~$\UD$. The situation for automorphisms is described in the example below.

\begin{example}\label{EX_lambda-family}
Let ${f\in\Aut(\UD)}$ and suppose that ${f(0)\neq0}$. As a direct calculation shows, $\Omega_\lambda(f)={\big\{z\colon|z-z_0|^2 > (|z_0|^2-1)/\lambda\big\}}$, where ${z_0:=-1/\overline{f(0)}}$. Consequently, $\Omega_\lambda(f)$ is h-convex if and only if~$\lambda\ge1$. Moreover, for $\lambda>1$, $\Omega_\lambda(f)$ is strictly h-convex, but this is not the case for $\Omega_1(f)$. In fact, $\UD\cap\partial\Omega_1(f)$ is precisely a geodesic.
\end{example}

Note that in the above example, $\Omega_\lambda(f)$ is the whole disk~$\UD$ if ${\lambda\ge(1+|f(0)|)/(1-|f(0)|)}$. At the end of this section, we will precisely characterise in which cases ${\Omega_\lambda(f)=\UD}$; see Proposition~\ref{PR_wholeDisk}.

\begin{proof}[\proofof{Theorem~\ref{TH_main1}}]
Observe that ${\Omega_\lambda(f)}={\big\{z\in\UD\colon u(z)>0\big\}}$, where
$$
  u(z):=|f(z)|^2-\lambda|z|^2+\lambda-1.
$$
It immediately follows that~${\Omega_\lambda(f)}$ is an open set. Moreover, it is known \cite[p.\,415]{AMP2019} that under the hypothesis of Theorem~\ref{TH_main1}, the set $\Omega_\lambda(f)$ is starlike with respect to the origin, hence connected. Note also that the function $u$ is real analytic throughout~$\UD$. Therefore, in order to apply  Lemma~\ref{LM_main}, it is sufficient to check that for every point ${\zeta\in\UD\cap\partial\Omega_\lambda(f)}$ the conditions \ref{IT_LM_main-grad-nonnullo}~and~\ref{IT_LM_main-geodesic} hold.

\StepPm{\ref{IT_LM_main-grad-nonnullo}.} We compute
$$
\nabla u(z) = {2\,\frac{\partial u}{\partial\hspace{.05em}\overline{\hspace{-.05em}z}} = 2 \left[f(z) \overline{f'(z)} - \lambda z \right]},
$$
so that ${\nabla u(\zeta)=0}$ is equivalent to $f(\zeta) \overline{f'(\zeta)} = \lambda \zeta$.  Moreover, ${1-|f(\zeta)|^2}={\lambda(1-|\zeta|^2)}$ because ${\zeta\in \partial \Omega_\lambda(f)}$. In particular, since ${\lambda\ge1}$, we have that $|f(\zeta)|\le|\zeta|$  with equality if and only if~${\lambda=1}$. Therefore, if ${\nabla u(\zeta)}$ vanished then using the Schwarz\,--\,Pick inequality~\eqref{EQ_SchP_inf} we would get
$$
   \lambda|\zeta|=|f(\zeta)f'(\zeta)|\le|f(\zeta)|\,\frac{1-|f(\zeta)|^2}{1-|\zeta|^2}
   =\lambda|f(\zeta)|\le\lambda|\zeta|,
$$
where at least one of the inequalities  would be strict unless ${\lambda=1}$ and ${f\in\Aut(\UD)}$. However, the latter possibility is excluded by the hypotheses.

\StepPm{\ref{IT_LM_main-geodesic}.} This part is more involved. We have to show that  ${u(z)<0}$ for all $z\in\gamma_\zeta\setminus\{\zeta\}$  close to~$\zeta$, where $\gamma_\zeta$ stands for the geodesic passing through~$\zeta$ orthogonally to~$\nabla u(\zeta)$.

Recall that $0\in\Omega_\lambda(f)$. Taking this into account and replacing $f$ by an appropriate rotation $e^{is}f(e^{it}z)$ we may assume without loss of generality  that ${\zeta>0}$ and ${f(\zeta)\geq0}$.  To further simplify the setting, we apply an automorphism ${\psi\in\Aut(\UD)}$ that takes $\zeta$ to~$0$. Then the geodesic~$\gamma_\zeta$ is mapped onto a diameter,  \textit{i.e.}~${\psi(\gamma_\zeta)}={(-\kappa,\kappa)}$ for some ${\kappa\in\UC}$. The inverse automorphism ${\v=\psi^{-1}}$ is given by
$$
\v(z) = \frac{\zeta + z}{ 1+ \zeta z} = \zeta +(1-\zeta^2)(z-\zeta z^2) + O(z^3), \qquad z\to0.
$$
Let $g:=f\circ\v$ and write
$$
g(z) = b + \sum_{n=1}^\infty b_n z^n, \qquad \text{where} \quad b=f(\zeta)\in [0,1).
$$
Let $v(t) := u\big( \v( t \kappa) \big)$, for ${t\in(-1,1)}$. It is clear that $v(0)=0$ since ${\zeta\in \partial\Omega_\lambda(f)}$. Also, we have that $v'(0)=0$ because $\gamma_\zeta$ is orthogonal to~$\nabla u(\zeta)$ at the point~$\zeta$. Our objective has now been reduced to showing that ${v''(0)<0}$.

From the definition of $u$ we get $v(t) = |g(t\kappa)|^2-\lambda|\v(t\kappa)|^2+\lambda-1$ and find the asymptotic expansion
\begin{align*}
v(t)  =&\, |b+b_1t\kappa + b_2 t^2\kappa^2|^2 - \lambda|\zeta +(1-\zeta^2)(t\kappa-\zeta t^2\kappa^2)|^2 + \lambda - 1 + O(t^3) \\
 = & \, b^2 - \lambda\zeta^2 + \lambda -1 + 2\Re\Big\{ \big[ b b_1 - \lambda \zeta (1-\zeta^2) \big]\kappa \Big\}\, t   \\
& + \bigg( |b_1|^2 -\lambda(1-\zeta^2)^2 + 2\Re\Big\{ \big[ b b_2 + \lambda \zeta^2 (1-\zeta^2) \big]\kappa^2 \Big\} \bigg) t^2 + O(t^3), \quad t\to0.
\end{align*}
Note that $v(0)=0$ is equivalent to
\begin{equation} \label{v-zero-is-zero}
1-b^2 = \lambda(1-\zeta^2).
\end{equation}
Let $A := b b_1 - \lambda \zeta (1-\zeta^2) $. Then $v'(0)=0$ is equivalent to
\begin{equation} \label{v-deriv-zero-is-zero}
\Re (A \kappa) =0.
\end{equation}
Consider the holomorphic function $h:\UD\to\overline\UD$ given by
$$
   h(z) := \frac{b-g(z)}{z \big( 1-bg(z)\big)} = \sum_{n=0}^\infty c_n z^n,\quad z\in\UD.
$$
Note that  $h(\UD)\subs\UD$ if ${f\not\in\Aut(\UD)}$, and ${h\equiv\const\in\UC}$ if ${f\in\Aut(\UD)}$. Therefore,
\begin{equation}\label{EQ_c-ests}
 |c_0|\le 1  \qquad \text{and} \qquad |c_1|\le 1-|c_0|^2,
\end{equation}
 the latter by the Schwarz\,--\,Pick inequality~\eqref{EQ_SchP_inf}, while the equality $|c_0|=1$ occurs if and only if $f\in\Aut(\UD)$. Moreover, as we have already seen,
\begin{equation*}
 b=f(\zeta)\le\zeta,
\end{equation*}
with  the inequality being strict unless~$\lambda=1$.
 Since in case ${\lambda=1}$ we suppose that $f\not\in\Aut(\UD)$, it follows that
\begin{equation}\label{EQ_bc0-zeta}
  b|c_0|<\zeta.
\end{equation}

Elementary calculations show that
\begin{equation} \label{b12-c01}
b_1 = -(1-b^2)c_0 \qquad \text{and} \qquad b_2 = -(1-b^2)(c_1+b c_0^2).
\end{equation}
Substituting in the expression for $A$ and using \eqref{v-zero-is-zero} we obtain
\begin{equation} \label{A-simple}
A = -(1-b^2)(b c_0 + \zeta).
\end{equation}
Combined with~\eqref{EQ_bc0-zeta}, this equality  implies that {$A\neq0$}. In view of \eqref{v-deriv-zero-is-zero}, we may write ${A \kappa = i \rho}$ for some ${\rho\in \R\backslash\{0\}}$. Clearly, $\rho^2=|A|^2$. Hence, using~\eqref{A-simple} we get
\begin{equation}\label{EQ_kappa}
\kappa^2 = -\frac{\rho^2}{A^2} = -\frac{\overline{A}}{A} =-\frac{b\overline{c_0}+\zeta}{ bc_0+\zeta }.
\end{equation}
Using \eqref{v-zero-is-zero} and \eqref{b12-c01} we get that the second coefficient in the asymptotic expansion for $v$ is
\begin{equation}\label{EQ_for-v-two-primes}
\frac{v''(0)}{2} = \lambda(1-\zeta^2)^2 ( \lambda |c_0|^2 - 1 )  + 2\Re\big\{ \lambda(1-\zeta^2) ( \zeta^2 -b c_1 -b^2 c_0^2 ) \kappa^2 \big\} \le \lambda(1-\zeta^2)\Phi,
\end{equation}
where
$$
\Phi:=(1-\zeta^2) ( \lambda |c_0|^2 - 1 ) + 2\Re\big\{ (\zeta^2 - b^2 c_0^2) \kappa^2 \big\} + 2b|c_1|.
$$
In view of \eqref{EQ_kappa} we have that
$$
\Re\{ (\zeta^2-b^2c_0^2) \kappa^2\} = - \Re\{ (\zeta-bc_0) (b\overline{c_0}+\zeta)\} = b^2|c_0|^2-\zeta^2.
$$
Now, using \eqref{EQ_c-ests}, we obtain
\begin{align*}
\Phi & =  (1-\zeta^2) ( \lambda |c_0|^2 - 1 ) + 2b^2|c_0|^2 -2\zeta^2 + 2b|c_1| \\
     &\le (1-\zeta^2) ( \lambda |c_0|^2 - 1 ) + 2b^2|c_0|^2 -2\zeta^2 + 2b(1-|c_0|^2)\\
     & = (1-b)^2 |c_0|^2 + 2b -1 -\zeta^2 \\
     & \leq  (1-b)^2 + 2b -1 -\zeta^2 \\
     & = b^2 - \zeta^2,
\end{align*}
with the last inequality  being strict unless~$f\in\Aut(\UD)$. Recalling that ${\zeta>b}$ if~${\lambda>1}$, and ${\zeta=b}$ if~${\lambda=1}$, we conclude that~${\Phi<0}$. Thus also ${v''(0)<0}$ as desired.
\end{proof}

\begin{proposition}\label{PR_wholeDisk}
Let $\lambda>0$ and let $f:\UD\to\UD$ be holomorphic. Then ${\Omega_\lambda(f)=\UD}$ if and only if the following two conditions hold:
\ifdefined\Pavel\begin{alphlist}\else\begin{enumerate}\fi
\item\label{IT_Blaschke}
   $f$ is a finite Blaschke product, i.e.
   $$
      f(z)=e^{i\theta}\prod_{k=1}^{n}\frac{z-a_k}{1-\overline{a_k}z},\quad z\in\UD,
   $$
   for some $\theta\in\Real$, $n\in\Natural$, and some points $a_1,\ldots a_n\in\UD$.

\medskip
\item\label{IT_lambda} $\lambda\ge\sup\limits_{z\in\UD}|f'(z)|$ and $\lambda>1$.
\ifdefined\Pavel\end{alphlist}\else\end{enumerate}\fi
\end{proposition}

\begin{proof}
\StepG{Necessity}
Assume that ${\Omega_\lambda(f)=\UD}$, \textit{i.e.}
\begin{equation}\label{EQ_wholeDisk}
\nu_f(z):=\frac{1-|f(z)|^2}{1-|z|^2}<\lambda\quad\text{for all~}~z\in\UD.
\end{equation}
Then in view of the Schwarz\,--\,Pick Lemma, ${|f'|<\lambda}$ in~$\UD$. In particular, $f$ extends continuously to the closure of~$\UD$. Appealing again to~\eqref{EQ_wholeDisk}, we have that ${|f(\zeta)|=1}$ for every ${\zeta\in\UC}$. Thus, $f$ must be a finite Blaschke product by a theorem of Fatou; see \textit{e.g.} \cite[Theorem~3.5.2]{BProd-book}. To see that ${\lambda>1}$ in this case, notice that
$$
  2\pi n~=~\int_{0}^{2\pi}|f'(e^{it})|\,\di t~\le~2\pi\lambda,
$$
where $n$ is the degree of the Blaschke product~$f$. It follows that ${\lambda\ge n\ge 1}$. If ${\lambda=1}$, then we get that ${f\in\Aut(\UD)}$ and ${|f'|<1}$ in~$\UD$, which is impossible.

\StepG{Sufficiency}
Suppose conditions \ref{IT_Blaschke} and~\ref{IT_lambda} hold. To show that the domain ${\Omega_\lambda(f)}$ coincides with~$\UD$, assume on the contrary that ${\Omega_\lambda(f)\neq\UD}$. Recall that $\Omega_\lambda(f)$ is not empty because ${0\in\Omega_\lambda(f)}$ when ${\lambda>1}$. Take any geodesic $\gamma$ tangent to $\partial\Omega_{\lambda}(f)$. Since $\Omega_\lambda(f)$ is h-convex by Theorem~\ref{TH_main1}, we have that $\Omega_\lambda(f)$ is contained in one of the two connected components of~$\UD\setminus\gamma$. Let $U$ stand for the other connected component. Fix an open arc ${\Gamma\subs\UC}$  contained in~$\partial U$. Since $f$ maps $\UC$ into~$\UC$, it is easy to see (using, \textit{e.g.}, Remark~\ref{RM_JWC}) that ${\zeta f'(\zeta)/f(\zeta)>0}$ and that $|f'(\zeta)|=\lim_{r\to1^-}\nu_f(r\zeta)$ for all ${\zeta\in\UC}$. Therefore, $|f'(\zeta)|\ge\lambda$ for all ${\zeta\in\Gamma}$. At the same time, ${|f'(\zeta)|\le\lambda}$ for all ${\zeta\in\UC}$ by condition~\ref{IT_lambda}. As a result, ${\zeta f'(\zeta)/f(\zeta)=\lambda}$ for all~${\zeta\in\Gamma}$. Thus, the rational
 function $z\mapsto zf'(z)/f(z)$ is constant and equal to~$\lambda$, which is only possible if ${f(z)=e^{i\theta}z^n}$ and ${\lambda=n}$ for some ${n\in\Natural}$ and some ${\theta\in\Real}$. Recalling that ${\lambda>1}$ by the hypothesis, it is easy to see that in this case~\eqref{EQ_wholeDisk} holds, contradicting  our assumption.
\end{proof}

\begin{remark}
Note that if at every point~$\zeta$ of some open arc ${\Gamma\subset\UC}$ a holomorphic self-map $f:\UD\to\UD$ has radial limit $\lim_{r\to1^-} f(r\zeta)$ belonging to~$\UC$ and attained uniformly with respect to~${\zeta\in\Gamma}$, then $f$~extends holomorphically to~$\Gamma$ by the Schwarz Reflection Principle. Combining this fact with  the Julia\,--\,Wolff\,--\,Carath\'eodory Theorem (see Remark~\ref{RM_JWC}), it is possible to establish a sort of local version of the above Proposition~\ref{PR_wholeDisk}: \textit{an open arc $\Gamma\subset\UC$ is contained in the boundary of~$\Omega_\lambda(f)$, ${\lambda>1}$, if and only if the function~$f$ admits a holomorphic extension to~$\Gamma$ satisfying $f(\zeta)\in\UC$ and $|f'(\zeta)|\le\lambda$ for all~${\zeta\in\Gamma}$}.
Furthermore, since $\liminf_{z\to\zeta}\nu_f(z)=\exp\big[\liminf_{z\to\zeta}k_{\UD}(z,0)-k_{\UD}(f(z),0)\big]$, see \textit{e.g.} \cite[Propositions~2.1.15 and~2.1.21]{Abate2}, a similar statement holds for the level sets $\mathcal D_\mu(f)$ studied in the next section.
\end{remark}

\section{Level sets defined by the hyperbolic distance}\label{S_hyperbolic-distance-family}
In this section we consider the sublevel sets
$$
 \mathcal D_\mu(f) = \mathcal D_\mu(f;z_0,w_0):=\big\{z\in \UD\colon k_\UD\big(z,z_0\big)-k_\UD\big(f(z),w_0\big)<\mu\big\},\quad \mu\in\Real,
$$
where $f:\UD\to\UD$ is  a holomorphic function and $z_0$, $w_0$ are points in~$\UD$.  Setting $g:={T_2^{-1}\circ f\circ T_1}$, for ${T_1,T_2\in\Aut(\UD)}$ that map the origin to $z_0$ and $w_0$, respectively, we see that $\mathcal D_\mu(f;z_0,w_0) =  T_1 \big(\mathcal D_\mu(g;0,0) \big).$ Recalling that automorphisms  preserve h-convexity,  in what follows we assume that
$$
z_0=w_0=0.
$$
Moreover, taking into account that  $\mathcal D_0(f)=\Omega_1(f)$ for this choice of $z_0$ and~$w_0$, we may exclude $\mu=0$ from consideration.

Our objective is to prove the following.

\begin{theorem}\label{TH_main2}
If $\mu<0$ and ${f:\UD\to\UD}$ is holomorphic then the set $\mathcal D_\mu(f)$ is either empty or a strictly hyperbolically convex domain.
\end{theorem}

This is sharp since for $\mu>0$ the following example shows that h-convexity may fail.

\begin{example}\label{EX_mu>0}
Let $\mu>0$ and set $b:=\tanh(\mu/2)\in(0,1)$. Consider the disk-automorphism ${f(z):=(z-b)/(1-bz)}$, ${z\in\UD}$. Note that $z\in\mathcal D_\mu(f)$ if and only if ${b-|z|+(1-b|z|)|f(z)|>0}$. Since $|f(z)|\ge f(|z|)$ for all ${z\in\UD}$, with equality if and only if~$z\in[b,1)$, it is easy to see that in this case, $\mathcal D_\mu(f)=\UD\setminus[b,1)$, which is not a hyperbolically convex domain.
\end{example}

 Observe that we may rewrite the defining inequality of the set $\mathcal D_\mu(f)$ as
\begin{align}\label{EQ_euclidian_Omega}
   \mathcal D_\mu(f)\,&=\,\Big\{ z\in\D\colon \frac{|z|-|f(z)|}{1-|zf(z)|} < \tanh(\mu/2) \Big\}\,=\,
   \big\{z\in\D\colon u(z)>0\big\},\\
\intertext{where}\label{EQ_def-of-u}
   u(z)&:= (1+a|z|) |f(z)| - |z| - a \qquad\text{and}\qquad a:= -\tanh(\mu/2).
\end{align}

We begin with some basic considerations in the following proposition.
\begin{proposition}\label{LM_nonempty}
Let $\mu<0$ and $f$ be a holomorphic self-map of $\UD$. Then ${\mathcal D_\mu(f)\neq\emptyset}$ if and only if ${|f(0)|> -\tanh(\mu/2)}$, in which case $\mathcal D_\mu(f)$ is a proper subset of $\UD$ containing the origin.
\end{proposition}
\begin{proof}
If $|f(0)|>a= -\tanh(\mu/2)$, then by~\eqref{EQ_euclidian_Omega} it is immediate that ${0\in\mathcal D_\mu(f)}$.
The  reverse inequality ${|f(0)|\le -\tanh(\mu/2)}$  is equivalent to ${k_\UD(f(0),0)\le-\mu}$. In such a case, applying the triangle inequality and the Schwarz\,--\,Pick Lemma  \cite[Theorem~3.2]{HBmetric1}, for all ${z\in\UD}$ we obtain
$$
 k_\UD(f(z),0)\,\le\, k_\UD(f(z),f(0))+k_\UD(f(0),0)\le k_\UD(z,0) -\mu,
$$
which shows that $\mathcal D_\mu(f)=\emptyset$.

To see that $\mathcal D_\mu(f)$ is not the whole unit disk when $|f(0)|>a$, it remains to recall that ${\mathcal D_\mu(f)\subset\mathcal D_0(f)=\Omega_1(f)}$ and to use Proposition~\ref{PR_wholeDisk}.
\end{proof}

In the proof of Theorem~\ref{TH_main2} we will use the following lemma.
\begin{lemma}\label{LM_gradient}
Let $\mu<0$ and $f:\UD\to\UD$  be holomorphic with ${|f(0)|> -\tanh(\mu/2)}$.  Then for every ${\zeta\in\partial\mathcal D_\mu(f) \cap \UD}$ we have ${|f(\zeta)|>|\zeta|>0}$ and
\begin{equation} \label{EQ_grad-projection}
 \Re\big\{\overline \zeta\,\nabla u(\zeta)\big\} < 0,
\end{equation}
where $u$ is defined by~\eqref{EQ_def-of-u}.

In particular, $\mathcal D_\mu(f)$ is a domain that is starlike with respect to the origin.
\end{lemma}
\begin{proof}
Since $\mu<0$, we have that $a =-\tanh(\mu/2)\in(0,1)$. Let ${\zeta\in\partial\mathcal D_\mu(f)\cap \UD}$. Note that ${\zeta\neq0}$ by Proposition~\ref{LM_nonempty}. Using the fact that ${u(\zeta)=0}$, we obtain
\begin{equation}\label{EQ_abs_f_of_zeta}
 |f(\zeta)|=\frac{|\zeta|+a}{1+a|\zeta|}\,>\,|\zeta|.
\end{equation}
In particular, $f(\zeta)\neq0$.

Note that $u(z)$ is real-analytic  whenever $z\neq0$ and $f(z)\neq0$. Employing the formula ${\nabla u(z)=2\, {\partial u(z)}/{\partial\hspace{.05em}\overline{\hspace{-.05em}z}}}$, we compute
$$
 \Re\big\{\overline z\,\nabla u(z)\big\}\,=\,
    |z|(a|f(z)|-1)\,+\,(1+a|z|)\,\Re\Big\{\frac{f(z)}{|f(z)|}\overline{zf'(z)}\Big\}.
$$
Taking into account~\eqref{EQ_abs_f_of_zeta} and using~\eqref{EQ_SchP_inf}, we therefore obtain
\begin{align*}
  \Re\big\{\overline \zeta\,\nabla u(\zeta)\big\}
     & = \frac{|\zeta| (a^2-1 ) }{1+a|\zeta|} \,+\, (1 + a|\zeta|) \Re\Big\{ \frac{f(\zeta)}{|f(\zeta)|}\,\overline{\zeta f'(\zeta)}\Big\} \\
     & \le \frac{|\zeta| (a^2-1 ) }{1+a|\zeta|} \,+\,  (1 + a|\zeta|)|\zeta f'(\zeta)|\\
     & \le~0,
\end{align*}
with equalities occurring in both inequalities if and only if ${f\in\Aut(\UD)}$ and ${\overline{\zeta f'(\zeta)} f(\zeta)>0}$. The only automorphisms satisfying  this last inequality and~\eqref{EQ_abs_f_of_zeta} are rotations of
\begin{equation*}
 f_0(z):=\frac{z+z_0}{1+\overline{z_0}z},\quad z_0:=a\frac{\zeta}{|\zeta|}.
\end{equation*}
However, $|f_0(0)|=a$, which contradicts the hypothesis of the lemma. Thus, the desired strict inequality holds.

Recall that $0\in\mathcal D_\mu(f)$ by Proposition~\ref{LM_nonempty}. Viewing the quantity in~\eqref{EQ_grad-projection} as the scalar product of the vectors~$\zeta$  and $\nabla u(\zeta)$ we get that the angle between them lies in $(\pi/2, 3\pi/2)$.
 Hence with the help of a standard argument,  it follows that $\mathcal D_\mu(f)$ is a starlike domain.
\end{proof}

\begin{proof}[\proofof{Theorem~\ref{TH_main2}}]
Let ${\mu<0}$ and let $f$  be a holomorphic self-map of $\UD$ such that ${\mathcal D_\mu(f)\neq\emptyset}$. Fix an arbitrary ${\zeta\in\UD\cap\partial\mathcal D_\mu(f)}$.  Using rotations, we may assume without loss of generality that ${\zeta>0}$ and ${b:=f(\zeta)>0}$.

In view of~\eqref{EQ_euclidian_Omega}, Lemmas~\ref{LM_main} and~\ref{LM_gradient}, in order to show that $\mathcal D_\mu(f)$ is strictly h-convex, working exactly as in the proof of Theorem~\ref{TH_main1}, it suffices to prove the following claim.

\medskip\noindent\textsc{Claim.} \textit{
Let $v(t) := u\big( \v( t \kappa) \big)$, ${t\in(-1,1)}$, where $\v(z) := {(\zeta + z)/(1 + \zeta z)}$ and ${\kappa\in\UC}$ is chosen in such a way that ${v'(0)=0}$. Then $v''(0)<0$.}
\medskip

Following the proof of Theorem~\ref{TH_main1}, we set $g:=f\circ \v$ and express $v'(0)$ and $v''(0)$  in terms of $b$, $\zeta$ and the first two Taylor coefficients of
$$
   h(z) := \frac{b-g(z)}{z \big( 1-bg(z)\big)} = \sum_{n=0}^\infty c_n z^n,\quad z\in\UD,
$$
which satisfy
\begin{equation}\label{EQ_coefficients}
   |c_0|\le1 \quad\text{and}\quad  |c_1|\le 1-|c_0|^2.
\end{equation}
Furthermore, the relation $0=v(0) = (1+a\zeta)b -\zeta -a$ allows us to eliminate the parameter~$a$. With an asymptotic expansion for $t\to0$ we see that $\di [u \circ \varphi(t\tilde\kappa) ]/\di t\,\big|_{t=0}=\Re(A\tilde\kappa)$, for arbitrary $\tilde\kappa\in\UC$, where
\begin{equation}\label{EQ_A}
  A := - \frac{(1-\zeta^2)(1-b^2)(c_0 +1)}{1-b\zeta}.
\end{equation}
On the other hand, we get $\di [u \circ \varphi(t\tilde\kappa) ]/\di t\,\big|_{t=0}=\Re [ \overline{ \nabla u(\zeta)} \varphi'(0) \tilde\kappa ]$ by the chain rule. Since $\nabla u(\zeta)\neq0$ by Lemma~\ref{LM_gradient}, it follows that $A=\overline{ \nabla u(\zeta)} \varphi'(0)\neq 0$, thus $c_0\neq-1$.

Similarly, a laborious but elementary calculation leads to
\begin{multline}\label{EQ_second}
\frac{v''(0)}2~=~ \frac{(1-\zeta^2)(1-b^2)}{1-b\zeta} \Bigg[ \frac{ (1-b^2) \big(\Im(c_0\kappa)\big)^2 }{2b} - \Re\big\{ (c_1+bc_0^2) \kappa^2\big\}     \\
  - (b-\zeta) \Re(\kappa) \Re(c_0 \kappa) +\zeta  -\frac{(1+3\zeta^2)\big(\Im\kappa\big)^2 }{2\zeta}  \Bigg].
\end{multline}

Since $\Re (A \kappa)=v'(0)=0$, we may write $A\kappa = i \rho$ for a suitable ${\rho>0}$ (replacing, if necessary, $\kappa$ by~$-\kappa$). We set
$$
c_0 +1 =: r e^{i\theta}, \qquad \text{with} \qquad |\theta|<\frac{\pi}{2}, \quad 0<r\leq 2\cos\theta.
$$
Taking into account~\eqref{EQ_A}, in this notation we have that $\kappa = \frac{i \rho }{A} = -i \frac{ |c_0+1| }{ c_0+1 } = -i e^{-i\theta}$. With the help of the equalities
\begin{align*}
&\kappa\,=\,-\sin\theta\,+\,i(-\cos\theta),  \qquad  |c_0|^2 = r^2 -2r\cos\theta+1,\\
&c_0\kappa\,=\,\sin\theta\,+\,i(\cos\theta-r), \quad\text{and}\quad \Re(c_0^2\kappa^2) = 1- 2\cos^2\theta +2r\cos\theta -r^2,
\end{align*}
formula~\eqref{EQ_second} can be rewritten as
$$
\Phi := \, \frac{1-b\zeta}{(1-\zeta^2)(1-b^2)} \cdot \frac{v''(0)}{2}%
= \,  \alpha r^2 - 2 \alpha r \cos\theta  + \beta \cos^2\theta - \Re (c_1 \kappa^2 ),
$$
where $\displaystyle \alpha := \frac{1+b^2}{2b} \quad \text{and} \quad  \beta := \frac{1+b^2}{2b} - \frac{1+\zeta^2}{2\zeta}$. \\[.75ex]According to~\eqref{EQ_coefficients}, we have $-\Re (c_1 \kappa^2 ) \le |c_1|  \le 1- |c_0|^2 = 2 r \cos\theta -r^2$. Hence,
$$
 \Phi~\le~ p(r):=(\alpha-1) r^2-2(\alpha-1) r\cos\theta+\beta\cos^2\theta\,=\,
 (\alpha-1) r(r-2\cos\theta)+\beta\cos^2\theta.
$$
Since $\alpha-1=(1-b)^2/(2b)>0$,  the quadratic polynomial $p$ is convex with respect to the variable $r$. Hence, it suffices to verify that $p$~is negative at the endpoints ${r=0}$ and ${r=2\cos \theta}$.
We have $p(0)={p(2\cos\theta)}=\beta\cos^2\theta$, with
 ${\beta<0}$, since ${0<\zeta<b<1}$ in view of Lemma~\ref{LM_gradient}.
Thus, $v''(0)<0$, which proves the claim.
\end{proof}

\begin{remark}
Recall that the hyperbolic geometry can be transferred from the unit disk to any simply connected hyperbolic domain via any of its Riemann mappings, carrying with it the notion of h-convexity. Hence, Theorem~\ref{TH_main2} proved above can be rewritten as follows. \textit{If ${\mu<0}$ and $f:D_1\to D_2$ is a holomorphic map between two simply connected hyperbolic domains ${D_j\subs\C}$, $j=1,2$, then for every ${z_1\in D_1}$ and every ${z_2\in D_2}$, the set
$$
\mathcal D_\mu(f;D_1,D_2):=   \{z\in D_1\colon k_{D_1}\big(z,z_1\big)-k_{D_2}\big(f(z),z_2\big)<\mu\},
$$
where $k_{D_j}$ stands for the hyperbolic distance in~$D_j$, is either empty or a strictly h-convex domain in~$D_1$. For ${\mu=0}$ the same statement holds\footnote{As for the case $\mu=0$, recall that it is covered by \cite[Theorem~1]{SolyninPAMS}; see also Theorem~\ref{TH_main1} and Example~\ref{EX_lambda-family}.} unless $f$ maps $D_1$ conformally onto $D_2$, in which case $\mathcal D_\mu(f;D_1,D_2)$ is merely h-convex (not strictly).}
\end{remark}

\section{Concluding remarks and open questions}\label{S_further}
{Theorem~\ref{TH_main2}, asserting the h-convexity of the sets $\mathcal D_\mu(f)$, $\mu\le0$ (where the case ${\mu=0}$ is due to Solynin~\cite{SolyninPAMS}), is stated in terms of hyperbolic geometry. However, our proof is entirely in euclidian terms.
\begin{problem}
  Give an intrinsic proof, \textit{i.e.}~a proof purely in terms of hyperbolic geometry, of the h-convexity of $\mathcal D_\mu(f)$, $\mu\le0$.
\end{problem}
We expect that such a proof could reveal some new interesting relationship between the hyperbolic metric and holomorphic mappings. Moreover, having a hyperbolic-geometric proof one may hope that Theorem~\ref{TH_main2} can be extended to a more general setting.  To be more concrete, we need to recall some definitions.

Recall that a \textsl{geodesic} in a metric space ${(X,k)}$ is an isometry $\gamma:I\to X$ from an interval ${I\subset\Real}$, endowed with the euclidian distance ${k_\Real(s,t):=|t-s|}$, to ${(X,k)}$. As usual, the image~$\gamma(I)$ of such an isometry is also referred to as a geodesic. The metric space ${(X,k)}$ is said to be \textsl{geodesic} if every pair of distinct points in~$X$ can be joined by a geodesic. Note that, in general, such a geodesic does not have to be unique. Moreover, in addition to these geodesics, there can exist also \textsl{local geodesics} joining the same two points. Recall that ${\gamma:I\to X}$ is called a local geodesic\footnote{It is worth mentioning that in Riemannian geometry the terminology is a bit different: geodesics are called \textsl{minimizers} or \textsl{length minimizing geodesics}, and local geodesics are referred to simply as geodesics.} if every ${t_0\in I}$ is contained in a subinterval ${J\subs I}$ such that the restriction of~$\gamma$ to~$J$ is a geodesic.
 As a result, in a geodesic metric space there are four different conditions that can be considered as analogues of convexity in~$\R^n$:

\ifdefined\Pavel\begin{romlist}\else\begin{enumerate}\fi
  \item\label{IT1} for every $x_1,x_2\in A\subs X$, $x_1\neq x_2$, the set~$A$ contains a local geodesic joining~$x_1$~and~$x_2$;
  \item\label{IT2} for every $x_1,x_2\in A\subs X$, $x_1\neq x_2$, the set~$A$ contains a geodesic joining~$x_1$ and~$x_2$;
  \item\label{IT3} every geodesic with end-points in~$A$ is contained in~$A$;
  \item\label{IT4} every local geodesic with end-points in~$A$ is contained in~$A$.
\ifdefined\Pavel\end{romlist}\else\end{enumerate}\fi
Clearly, \ref{IT1} $\Leftarrow$ \ref{IT2} $\Leftarrow$ \ref{IT3} $\Leftarrow$ \ref{IT4}. Moreover, if a local geodesic joining two distinct points is unique, as we have \textit{e.g.} for a simply connected domain endowed with the hyperbolic distance, then the four conditions are equivalent. However, this is not true for multiply connected hyperbolic domains in~$\C$, with an annulus giving a simplest example. It is therefore natural to ask the following question.

\begin{problem}
Fix $\mu\in\Real$ and two hyperbolic (multiply connected) domains $D_j\subs\C$, $j=1,2$, equipped with the hyperbolic distance functions~$k_{D_j}$. Further, let ${z_j\in D_j}$, ${j=1,2}$. In this setting, which of the conditions \ref{IT1}\,-- \ref{IT4} hold for  $A:=\mathcal D_\mu(f;D_1,D_2)$ and for every holomorphic map ${f:D_1\to D_2}$?
\end{problem}

\begin{remark*}
Note that the above problem makes sense also in several complex variables. For example, one can ask the same question with $D_1$ and $D_2$ replaced by Kobayashi hyperbolic complex manifolds provided that the first manifold is a geodesic metric space with respect to the Kobayashi distance.\footnote{This is the case,  \textit{e.g.}, for bounded convex domains in~$\C^n$, see \cite[Theorem~2.6.19]{Abate1}; and for complete Kobayashi hyperbolic complex manifolds on which the Kobayashi metric is Finsler, as follows from the Hopf\,--\,Rinow Theorem, see \textit{e.g.} \cite[Corollary~1 in~\S8--7]{DiffGeom}.}
\end{remark*}

In contrast to the level sets~$\mathcal D_\mu(f)$, the hyperbolic-density family
$$
\Omega_\lambda(f):=\Big\{z\in\UD\colon \frac{\lambda_\UD(z)}{\lambda_\UD(f(z))}<\lambda\Big\}
$$
is not conformally invariant. Therefore, replacing~$\UD$ with another hyperbolic domain, even simply connected, may change the situation.

\begin{problem}
For which pairs of hyperbolic domains $D_j\subs\C$, $j=1,2$, the sets
$$
  \Omega_\lambda(f;D_1,D_2):=\Big\{z\in\UD\colon \frac{\lambda_{D_1}(z)}{\lambda_{D_2}(f(z))}<\lambda\Big\}
$$
satisfy, for every holomorphic map $f:D_1\to D_2$ and every $\lambda\ge1$, at least one of the conditions~\ref{IT1}\,-- \ref{IT4}? Here $\lambda_{D_j}$, ${j=1,2}$, stands for the density of the hyperbolic metric in~$D_j$.
\end{problem}

Returning to the original setting of the paper, we have another question, which seems to be of interest too.
\begin{problem}
  Fix some ${\mu\in\Real}$. For which strictly increasing functions $\Phi:[0,+\infty)\to\Real$ the sets $\mathcal D_{\Phi,\mu}(f)$ consisting of all ${z\in\UD}$ satisfying
  $$
    \Phi\big(k_\UD(z,0)\big)\,-\,\Phi\big(k_\UD(f(z),0)\big)\,<\,\mu
  $$
  are hyperbolically convex whenever $f:\UD\to\UD$ is holomorphic?
\end{problem}
Clearly, for $\mu=0$, every strictly increasing function~$\Phi$ defines the same hyperbolically convex set ${\mathcal D_{\Phi,0}(f)=\mathcal D(f)}$. Furthermore, our results show that the functions $\Phi_{-}(x):=x$ if ${\mu<0}$ and $\Phi_{+}(x):=\log\cosh(x/2)$ if~${\mu>0}$ are suitable choices. Given some fixed~${\mu\neq0}$, does there exist any other~$\Phi$ for which $\mathcal D_{\Phi,\mu}(f)$ is hyperbolically convex? Are there any choices valid for all $\mu$ in some open interval rather than for a single fixed value, aside from multiples of~$\Phi_{-}$ and~$\Phi_{+}$?

We conclude this section with a question related to the notion of geodesic convexity of a \textit{real-valued function}. Recall that a function ${u:X\to\Real}$ in a geodesic metric space~$(X,k)$ is said to be \textsl{geodesically} (or \textsl{geodetically}) \textsl{convex} if  for every geodesic~${\gamma\subs X}$ the function $u\circ \gamma$ of a real variable is convex.

It is easy to see that the sublevel sets of a geodesically convex function~$u$ are geodesically convex sets, actually in the strongest sense~\ref{IT4}. In this respect, it is worth noticing the following: our argument in the proof of Theorem~\ref{TH_main1} shows that for every ${\lambda>1}$, every holomorphic self-map ${f:\UD\to\UD}$ and every point ${\zeta\in\partial\Omega_{\lambda}(f)}$, the function ${u_1(z):= \lambda|z|^2-|f(z)|^2}$, restricted to the geodesic $\gamma_\zeta$ tangent at~$\zeta$ to ${\partial\Omega_{\lambda}(f)}$, is convex in a neighbourhood of~$\zeta$ because $(u_1\circ\gamma_\zeta)''$ is positive at the point~$\gamma_{\zeta}^{-1}(\zeta)$. A similar observation applies to the function $u_2(z):= |z| -(1+a|z|)|f(z)|$, ${a\in(0,1)}$,  near boundary points of~${\mathcal D_\mu(f)}$ with $\mu:={\log\big((1-a)/(1+a)\big)}$.

\begin{problem}
Is it possible to establish Theorems~\ref{TH_main1} and~\ref{TH_main2} by proving geodesic convexity of $u(z):=F\big(z,f(z)\big)$ in~$(\UD,k_\UD)$ for a suitable choice of ${F:\UD^2\to\Real}$?
\end{problem}


\end{document}